\documentclass{amsart}

\usepackage{amsmath,amsfonts,amssymb,amsmath,latexsym,multirow}

\newtheorem{theorem}{Theorem}[section]
\newtheorem{lemma}[theorem]{Lemma}
\newtheorem{corollary}[theorem]{Corollary}
\newtheorem{proposition}[theorem]{Proposition}

\theoremstyle{definition}

\theoremstyle{remark}
\newtheorem{remark}[theorem]{Remark}

\numberwithin{equation}{section}

\begin{document}

\title[Kloosterman Sums with Trace Nonzero Square Arguments]{$\begin{array}{c}
         \text{Ternary Codes Associated with $O(3,3^r)$ and }\\
           \text{Power Moments of Kloosterman Sums with}\\
           \text{Trace Nonzero Square Arguments}
       \end{array}$
}

\author{dae san kim}
\address{Department of Mathematics, Sogang University, Seoul 121-742, Korea}
\curraddr{Department of Mathematics, Sogang University, Seoul
121-742, Korea} \email{dskim@sogong.ac.kr}
\thanks{This work was supported by National Research Foundation of Korea Grant funded by
the Korean Government 2009-0072514.}

\subjclass[2000]{}

\date{}

\dedicatory{ }


\begin{abstract}
In this paper, we construct two ternary linear codes $C(SO(3,q))$
and $C(O(3,q))$, respectively associated with the orthogonal groups
$SO(3,q)$ and $O(3,q)$. Here $q$ is a power of three. Then we obtain
two recursive formulas for the power moments of  Kloosterman sums
with $\lq\lq$trace nonzero square arguments" in terms of the
frequencies of weights in the codes. This is done via Pless power
moment identity and by utilizing the explicit expressions of Gauss
sums for the orthogonal groups.\\

Index terms - power moment,  Kloosterman sum,  trace  nonzero square
argument, orthogonal group, Pless power moment identity, weight
distribution, Gauss sum.\\

\end{abstract}

\maketitle
\section{Introduction}

Let $\psi$ be a nontrivial additive character of the finite field
$\mathbb{F}_q$ with $q=p^r$ elements ($p$ a prime). Then the
 Kloosterman sum $K(\psi ; a)$(\cite{LN1}) is defined by

\begin{align*}
K(\psi ; a)=\sum_{\alpha \in \mathbb{F}_q^*} \psi(\alpha + a
\alpha^{-1})~(a \in \mathbb{F}_q^*).
\end{align*}
The Kloosterman sum was introduced in 1926(\cite{K1}) to give an
estimate for the Fourier coefficients of modular forms.

For each nonnegative integer $h$, by $MK(\psi)^h$ we will denote the
$h$-th moment of the Kloosterman sum $K(\psi;a)$. Namely, it is given by

\begin{align*}
MK(\psi)^h=\sum_{\alpha \in \mathbb{F}_q^*}K(\psi;a)^h.
\end{align*}
If $\psi=\lambda$ is the canonical additive character of
$\mathbb{F}_q$, then $MK(\lambda)^h$ will be simply denoted by
$MK^h$.

Explicit computations on power moments of Kloosterman sums were
begun with the paper {\cite{S1}} of Sali$\acute{e}$  in 1931, where
he showed, for any odd prime $q$,

\begin{align*}
MK^h=q^2 M_{h-1}-(q-1)^{h-1}+2(-1)^{h-1}~(h\geq1).
\end{align*}
Here $M_0=0$, and, for $h \in \mathbb{Z}_{>0}$,
\begin{align*}
M_h=|\{(\alpha_1,\cdots,\alpha_h)\in
(\mathbb{F}_q^*)^h|\sum_{j=1}^{h}\alpha_j=1=\sum_{j=1}^h
\alpha_j^{-1}\}|.
\end{align*}
For $q=p$ odd prime, Sali$\acute{e}$ obtained $MK^1$, $MK^2$,
$MK^3$, $MK^4$ in \cite{S1} by determining $M_1$, $M_2$, $M_3$. On
the other hand, $MK^5$ can be expressed in terms of the $p$-th
eigenvalue for a weight $3$ newform on $\Gamma_0(15)$ (cf.
\cite{L1}, \cite{PV1}). $MK^6$ can be expressed in terms of the
$p$-th eigenvalue for a weight $4$ newform on $\Gamma_0(6)$
(cf.\cite{HS1}). Also, based on numerical evidence, in \cite{E1}
Evans was led to propose a conjecture which expresses $MK^7$ in
terms of Hecke eigenvalues for a weight $3$ newform on
$\Gamma_0(525)$ with quartic nebentypus of conductor 105.

From now on, let us assume that $q=3^r$. Recently, Moisio was able
to find explicit expressions of $MK^h$, for $h\leq10$ (cf.
\cite{M1}). This was done, via Pless power moment identity, by
connecting moments of Kloosterman sums and the frequencies of
weights in the ternary Melas code of length $q-1$, which were
known by the work of Geer, Schoof and Vlugt in \cite{GS1}.

In order to describe our results, we introduce three incomplete
power moments of Kloosterman sums. For every nonnegative integer
$h$, and $\psi$ as before, we define

\begin{equation}\label{a}
T_0 SK(\psi)^h=\sum_{a\in\mathbb{F}_{q}^{*},~tra=0}K(\psi;a^2)^h,~
T_{12}
SK(\psi)^h=\sum_{a\in\mathbb{F}_{q}^{*},~tra\neq0}K(\psi;a^2)^h,
\end{equation}
which will be respectively  called the $h$-th moment of Kloosterman
sums with $\lq\lq$trace zero square arguments" and those with
$\lq\lq$trace nonzero square arguments." Then, clearly we have
\begin{equation}\label{b}
2SK(\psi)^h=T_0SK(\psi)^h+T_{12}SK(\psi)^h,
\end{equation}
where
\begin{equation}\label{c}
SK(\psi)^h=\sum_{a\in\mathbb{F}_{q}^{*},~a~square}K(\psi;a)^h,
\end{equation}
which is called the $h$-th moment of Kloosterman sums with
$\lq\lq$square arguments." If $\psi=\lambda$ is the canonical
additive character of $\mathbb{F}_{q}^{}$, then $SK(\lambda)^h$,
$T_{0}SK(\lambda)^h$, and $T_{12}SK(\lambda)^h$ will be respectively
denoted by $SK^h$, $T_{0}SK^h$ and $T_{12}SK^h$, for brevity.

We derived recursive formulas generating the odd power moments of
Kloosterman sums with trace one arguments in \cite{D2} and
\cite{D5}. To do that we constructed binary linear codes associated
with $O(3,2^r)$ and with double cosets with respect to certain
maximal parabolic subgroup of $O(2n+1,2^r)$.

 In this paper, we will show the main Theorem \ref{A} giving recursive
formulas for the power moments of Kloosterman sums with
$\lq\lq$trace nonzero square arguments." To do that, we construct
ternary linear codes $C(SO(3,q))$ and $C(O(3,q))$, respectively
associated with the orthogonal groups $SO(3,q)$ and $O(3,q)$, and
express those power moments in terms of the frequencies of weights
in the codes. Then, thanks to our previous results on the explicit
expressions of $\lq\lq$Gauss sums" for the orthogonal  group
$O(2n+1,q)$ \cite{D1}, we can express the weight of each codeword in
the duals of the codes in terms of Kloosterman sums. Then our
formulas will follow immediately from the Pless power moment
identity.

Henceforth, we agree that, for nonnegative integers $a,b,c$,
\begin{equation}\label{d}
{\binom{c}{a,b}}={\frac{c!}{a!~b!~(c-a-b)!}},~if~a+b\leq c,
\end{equation}
and
\begin{equation}\label{e}
{\binom{c}{a,b}}=0,~if~a+b>c.
\end{equation}
\indent

\begin{theorem}\label{A}
Let $q=3^r$. Then we have the following.

$(1)$ For $h=1,2,3,\cdots,$
\begin{align}\label{f}
\begin{split}
&((-1)^{h+1}+2^{-h})T_{12}SK^h\\
&=-\sum_{j=1}^{h-1}((-1)^{j+1}+2^{-j}){\binom{h}{j}}(q^2-1)^{h-j}T_{12}SK^j\\
&\quad+q^{1-h}\sum_{j=0}^{min\{N_1,h\}}(-1)^j(C_{1,j}-\hat{C}_j)\sum_{t=j}^{h}t!
S(h,t)3^{h-t}2^{t-h-j}{\binom{N_1-j}{N_1-t}},
\end{split}
\end{align}
where $N_1=|SO(3,q)|=q(q^2-1)$, and $\{C_{1,j}\}_{j=0}^{N_1}$ and
$\{\hat{C}_{j}\}_{j=0}^{N_1}$ are respectively the weight
distributions of $C(SO(3,q))$ and $C(Sp(2,q))$ given by: for
$j=0,\cdots,N_1$,
\begin{align}\label{g}
\begin{split}
C_{1,j}=&\sum{\binom{q^2}{\nu_0,\mu_0}}{\binom{q^2}{\nu_2,\mu_2}}\\
& \qquad \times\prod_{\beta^2-2\beta\neq0~square}
{\binom{q^2+q}{\nu_\beta,\mu_\beta}}
\prod_{\beta^2-2\beta~nonsquare}
{\binom{q^2-q}{\nu_\beta,\mu_\beta}},
\end{split}
\end{align}
\begin{align}\label{h}
\begin{split}
\hat{C}_{j}=&\sum{\binom{q^2}{\nu_1,\mu_1}}{\binom{q^2}{\nu_{-1},\mu_{-1}}}\\
& \qquad \times\prod_{\beta^2-1\neq0~square}
{\binom{q^2+q}{\nu_\beta,\mu_\beta}} \prod_{\beta^2-1~nonsquare}
{\binom{q^2-q}{\nu_\beta,\mu_\beta}}.
\end{split}
\end{align}
Here the first sum in (\ref{f}) is 0 if $h=1$ and the unspecified
sums in (\ref{g}) and (\ref{h}) run over all the sets of nonnegative
integers $\{\nu_\beta\}_{\beta\in\mathbb{F}_q}$ and
$\{\mu_\beta\}_{\beta\in\mathbb{F}_q}$ satisfying
\begin{align*}
\sum_{\beta\in\mathbb{F}_q}\nu_\beta+\sum_{\beta\in\mathbb{F}_q}\mu_\beta=j,~
and~ \sum_{\beta\in\mathbb{F}_q}\nu_\beta
\beta=\sum_{\beta\in\mathbb{F}_q}\mu_\beta \beta.
\end{align*}
In addition, $S(h,t)$ is the Stirling number of the second kind
defined by
\begin{align}\label{i}
S(h,t)=\frac{1}{t!}\sum_{j=0}^{t}(-1)^{t-j}{\binom{t}{j}}j^h.
\end{align}

$(2)$ For $h=1,2,3,\cdots,$

\begin{align}\label{j}
\begin{split}
&((-1)^{h+1}+2^{-h})T_{12}SK^h\\
&=-\sum_{j=1}^{h-1}((-1)^{j+1}+s^{-j}){\binom{h}{j}}(q^2-1)^{h-j}T_{12}SK^j\\
&\quad +q^{1-h}\sum_{j=0}^{min\{N_2,h\}}(-1)^j
C_{2,j}\sum_{t=j}^{h}t!S(h,t)3^{h-t}2^{t-2h-j}
{\binom{N_2-j}{N_2-t}}\\
&\quad -q^{1-h}\sum_{j=0}^{min\{N_1,h\}}(-1)^j
\hat{C}_{j}\sum_{t=j}^{h}t!S(h,t)3^{h-t}2^{t-h-j}
{\binom{N_1-j}{N_1-t}},
\end{split}
\end{align}
where $N_2=|O(3,q)|=2q(q^2-1)$, and $\{C_{2,j}\}_{j=0}^{N_1}$ is the
weight distribution of $C(O(3,q))$ given by: for $j=0,\cdots,N_2$,
\begin{align}\label{k}
\begin{split}
C_{2,j}=&\sum\prod_{\beta\in\mathbb{F}_q}{\binom{n_2(\beta)}{\nu_\beta,\mu_\beta}}
~(j=0,\cdots,N_2),\\
&\quad with~
n_2(\beta)=2q^2-2q+q\delta(1,q;\beta-1)+q\delta(1,q;\beta+1).
\end{split}
\end{align}
Here the first sum in (\ref{j}) is 0 if $h=1$, the unspecified sum
in (\ref{k}) runs over all the sets of nonnegative integers
$\{\nu_\beta\}_{\beta\in\mathbb{F}_q}$ and
$\{\mu_\beta\}_{\beta\in\mathbb{F}_q}$ satisfying
\begin{align*}
\sum_{\beta\in\mathbb{F}_q}\nu_\beta+\sum_{\beta\in\mathbb{F}_q}\mu_\beta=j,\quad
~and~ \sum_{\beta\in\mathbb{F}_q}\nu_\beta
\beta=\sum_{\beta\in\mathbb{F}_q}\mu_\beta \beta,
\end{align*}
$S(h,t)$ indicates the Stirling number of the second as in
(\ref{i}), $\hat{C}_j$'s are as in (\ref{h}), and
\begin{equation}\label{l}
\begin{split}
\delta(1,q;\beta)&=|\{x\in\mathbb{F}_{q}^{}|x^2-\beta x+1=0\}|\\
&=
\begin{cases}
2,& \text {if $\beta^2-1\neq0$ is a square,}\\
1,& \text {if $\beta^2-1=0$,}\\
0,& \text {if $\beta^2-1$ is a nonsquare.}
\end{cases}
\end{split}
\end{equation}
\end{theorem}
\indent

\section{$O^{}(2n+1,q)$}
\indent

For more details about the results of this section, one is referred
to the paper \cite{D1}. Throughout this paper, the following
notations will be used:
\begin{align*}
\begin{split}
q&=3^r~(r\in\mathbb{Z}_{>0}),\\
\mathbb{F}_q&=~the~finite~field~ with~ q~ elements,\\
TrA&=~the~ trace~ of~ A~ for~ a~ square~ matrix~ A,\\
^t B&=~the~transpose~ of~ B~for~any~matrix~B.
\end{split}
\end{align*}
\indent

The orthogonal group $O(2n+1,q)$ is defined as:
\begin{align*}
O(2n+1,q)=\{w\in GL(2n+1,q)|^twJw=J\},
\end{align*}
where
\begin{align*}
J=\begin{bmatrix}
    0 & 1_{n} & 0 \\
    1_{n} & 0 & 0 \\
    0 & 0 & 1 \\
  \end{bmatrix}.
\end{align*}
It consists of the matrices

\begin{align*}
  \begin{bmatrix}
    A & B & e \\
    C & D & f \\
    g & h & i \\
  \end{bmatrix}
(A,B,C,D~n\times n,e,f~n\times1,g,h~1\times n,i~1\times1)
\end{align*}
in $GL(2n+1,q)$ satisfying the relations: \indent

\begin{align*}
\begin{split}
&{^tA}C+{^tC}A+{^tg}g=0,~{^tB}D+{^tD}B+{^th}h=0,\\
&{^tA}D+{^tC}B+{^tg}h=1_{n},~{^te}f+{^tf}e+i^2=1,\\
&{^tA}f+{^tC}e+{^tg}i=0,~{^tB}f+{^tD}e+{^th}i=0.\\
\end{split}
\end{align*}
\indent

Let $P(2n+1,q)$ be the maximal parabolic subgroup of $O(2n+1,q)$
given by
\begin{align*}
\begin{split}
P&=P(2n+1,q)\\
&=\left\{\begin{bmatrix}
        A & 0 & 0 \\
        0 & {^tA}^{-1} & 0 \\
        0 & 0 & i \\
      \end{bmatrix}
       \begin{bmatrix}
               1_{n} & B & -{^th} \\
               0 & 1_{n} & 0 \\
               0 & h & 1 \\
             \end{bmatrix}
\Bigg|\begin{array}{c}
   A\in GL(n,q),~i=\pm1 \\
   B+{^tB}+{^th}h=0
 \end{array}
\right\},
\end{split}
\end{align*}
and let $Q=Q(2n+1,q)$ be the subgroup of $P(2n+1,q)$ of index 2
defined by
\begin{align*}
\begin{split}
Q&=Q(2n+1,q)\\
&=\left\{
      \begin{bmatrix}
        A & 0 & 0 \\
        0 & {^tA}^{-1} & 0 \\
        0 & 0 & 1 \\
      \end{bmatrix}
             \begin{bmatrix}
               1_{n} & B & -{^th} \\
               0 & 1_{n} & 0 \\
               0 & h & 1 \\
             \end{bmatrix}
\Bigg|\begin{array}{c}
   A\in GL(n,q) \\
   B+{^tB}+{^th}h=0
 \end{array}
\right\}.
\end{split}
\end{align*}
Then we see that
\begin{align*}
P(2n+1,q)=Q(2n+1,q)\amalg\rho Q(2n+1,q),
\end{align*}
with
\begin{align*}
\rho=\begin{bmatrix}
         1_{n} & 0 & 0 \\
         0 & 1_{n} & 0 \\
         0 & 0 & -1 \\
       \end{bmatrix}.
\end{align*}
\indent

Let $\sigma_r$ denote the following matrix in $O(2n+1,q)$
\begin{align*}
\sigma_r=\begin{bmatrix}
             0 & 0 & 1_r & 0 & 0 \\
             0 & 1_{n-r} & 0 & 0 & 0 \\
             1_r & 0 & 0 & 0 & 0 \\
             0 & 0 & 0 & 1_{n-r} & 0 \\
             0 & 0 & 0 & 0 & 1 \\
           \end{bmatrix}
~(0\leq r\leq n).
\end{align*}
Then the Bruhat decomposition of $O(2n+1,q)$ with respect to
$P=P(2n+1,q)$ is given by
\begin{align*}
O(2n+1,q)=\coprod_{r=0}^{n}P\sigma_r P=\coprod_{r=0}^{n}P\sigma_r Q,
\end{align*}
which can further be modified as
\begin{equation}\label{m}
\begin{split}
O(2n+1,q)&=\coprod_{r=0}^{n}P\sigma_r (B_r\setminus Q)\\
&=\coprod_{r=0}^{n}Q\sigma_r (B_r\setminus
Q)\amalg\coprod_{r=0}^{n}\rho Q\sigma_r(B_r\setminus Q),
\end{split}
\end{equation}
with
\begin{align*}
B_r=B_r (q)=\{w\in Q(2n+1,q)|\sigma_r w \sigma_r^{-1}\in
P(2n+1,q)\}.
\end{align*}
\indent

The special orthogonal group $SO(2n+1,q)$ is defined as
\begin{align*}
SO(2n+1,q)=\{w\in O(2n+1,q)|\text{det}w=1\}.
\end{align*}
Then we see from (\ref{m}) that
\begin{equation}\label{n}
SO(2n+1,q)=\coprod_{0\leq r\leq n,~r~even}Q\sigma_r(B_r\setminus
Q)\amalg\coprod_{0\leq r\leq n,~r~odd} \rho Q\sigma_r(B_r\setminus
Q).
\end{equation}
\indent

The sympletic group $Sp(2n,q)$ is defined as:
\begin{align*}
Sp(2n,q)=\{w\in GL(2n,q)|{^tw}\hat{J}{w}=\hat{J}\},
\end{align*}
with
\begin{align*}
\hat{J}=\begin{bmatrix}
            0 & 1_n \\
            1_n & 0 \\
          \end{bmatrix}.
\end{align*}
\indent

As is well-known or mentioned in \cite{D3} and \cite{D1},
\begin{equation}\label{o}
|O(2n+1,q)|=2q^{n^2}\prod_{j=1}^{n}(q^{2j}-1),\qquad\qquad\qquad
\end{equation}
\begin{equation}\label{p}
|SO(2n+1,q)|=|Sp(2n,q)|=q^{n^2}\prod_{j=1}^{n}(q^{2j}-1).
\end{equation}
\indent

For integers $n$, $r$ with $0\leq r\leq n$, the $q$-binomial
coefficients are defined as:
\begin{align*}
\begin{bmatrix}
                                                              n \\
                                                              r \\
                                                            \end{bmatrix}
_q=\prod_{j=0}^{r-1}(q^{n-j}-1)/(q^{r-j}-1).
\end{align*}
It is shown in \cite{D1} that
\begin{equation}\label{r}
|B_r(q)\setminus Q(2n+1)|=q^{\binom{r+1}{2}}
                                                            \begin{bmatrix}
                                                              n \\
                                                              r \\
                                                            \end{bmatrix}
_q.
\end{equation}
\indent

\section{Gauss sums for $O^{}(2n+1,q)$}
\indent

The following notations will be employed throughout this paper.

\begin{align*}
\begin{split}
tr(x)&=x+x^{3}+\cdots+x^{{3}^{r-1}} ~the~ trace~function~ \mathbb{F}_{q}~\rightarrow ~\mathbb{F}_{3},\\
\lambda_0(x)&=e^{2\pi ix/3} ~the~ canonical~ additive~ character~
of~\mathbb{F}_{3}^{},\\
\lambda(x)&=e^{2\pi itr(x)/3} ~the~ canonical~ additive~ character~
of~ \mathbb{F}_{q}^{}.
\end{split}
\end{align*}
Then any nontrivial additive character $\psi$ of $\mathbb{F}_q$ is
given by $\psi(x)=\lambda(ax)$, for a unique
$a\in\mathbb{F}_{q}^{*}$. Also, since $\lambda(a)$ for any
$a\in\mathbb{F}_q$ is a 3th root of 1, we have

\begin{equation}\label{s}
\lambda(-a)=\lambda(2a)=\lambda(a)^2=\lambda(a)^{-1}=\overline{\lambda(a)}.
\end{equation}
\indent

For any nontrivial additive character $\psi$ of $\mathbb{F}_q$ and
$a\in\mathbb{F}_{q}^{*}$, the Kloosterman sum $K_{GL(t,q)}(\psi;a)$
for $GL(t,q)$ is defined as

\begin{align*}
K_{GL(t,q)}(\psi;a)=\sum_{w\in GL(t,q)}\psi(Trw+aTrw^{-1}).
\end{align*}
Observe that, for $t=1$, $K_{GL(1,q)}(\psi;a)$ denotes the
Kloosterman sum $K(\psi;a)$.

\noindent In \cite{D3}, it is shown that $K_{GL(t,q)}(\psi;a)$
satisfies the following recursive relation: for integers $t\geq2$,
$a\in\mathbb{F}_q^*$,
\begin{align*}
K_{GL(t,q)}(\psi;a)=q^{t-1}K_{GL(t-1,q)}(\psi;a)K(\psi;a)+q^{2t-2}(q^{t-1}-1)
K_{GL(t-2,q)}(\psi;a),
\end{align*}
where we understand that $K_{GL(0,q)}(\psi;a)=1$.\\
\indent

\begin{proposition}\label{B}$($\cite{D1}$)$
Let $\psi$ be a nontrivial additive character of $\mathbb{F}_q$. For
each positive integer $r$, let $\Omega_r$ be the set of all $r\times
r$ nonsingular symmetric matrices over $\mathbb{F}_q$. Then we have
\begin{equation}\label{t}
a_r(\psi)=\sum_{B\in\Omega_r}\sum_{h\in\mathbb{F}_q^{r\times
1}}\psi({^th}Bh)=
\begin{cases}
q^{r(r+2)/4}\prod_{j=1}^{r/2}(q^{2j-1}-1), & \text{for $r$ even},\\
0,& \text{for $r$ odd.}
\end{cases}
\end{equation}
\end{proposition}
\indent

From \cite{D3} and \cite{D1}, the Gauss sums for $SO(2n+1,q)$ and
$O(2n+1,q)$ are respectively equal to $\psi(1)$ times that for
$Sp(2n,q)$ and $\psi(1)+\psi(-1)$ times that for $Sp(2n,q)$. Indeed,
using the decomposition in (\ref{n}), for any nontrivial additive
character $\psi$ of $\mathbb{F}_q$, it is shown that

\begin{align*}
\begin{split}
&\sum_{w\in SO(2n+1,q)}\psi(Trw)\\
&=\sum_{\substack{0\leq r \leq n\\r~even}}|B_r\setminus Q|\sum_{w\in
Q}\psi(Trw\sigma_r)
+\sum_{\substack{0\leq r \leq n\\r~odd}}|B_r\setminus Q|\sum_{w\in Q}\psi(Tr\rho w\sigma_r)\\
&=q^{\binom{n+1}{2}}\{\psi(1)\sum_{\substack{0\leq r \leq n\\r~even}}|B_r\setminus Q|
q^{r(n-r-1)}a_r(\psi)K_{GL(n-r,q)}(\psi;1)\\
&\qquad \qquad \qquad \qquad +\psi(-1)\sum_{\substack{0\leq r \leq
n\\r~odd}}|B_r\setminus Q|
q^{r(n-r-1)}a_r(\psi)K_{GL(n-r,q)}(\psi;1)\}\\
&=\psi(1)q^{\binom{n+1}{2}}\sum_{\substack{0\leq r \leq
n\\r~even}}q^{rn-{\frac{1}{4}}r^2}\begin{bmatrix}
                                                              n \\
                                                              r \\
                                                            \end{bmatrix}
_q\prod_{j=1}^{r/2}(q^{2j-1}-1)K_{GL(n-r,q)}(\psi;1)
~(cf.(\ref{r}),~(\ref{t}))\\
&(=\psi(1)\sum_{w\in Sp(2n,q)}\psi(Trw))~(cf. ~[4]).
\end{split}
\end{align*}
\indent

Similarly, from the decomposition in (\ref{m}) it is shown in
\cite{D1} that

\begin{align*}
\begin{split}
&\sum_{w\in O(2n+1,q)}\psi(Trw)\\
&=(\psi(1)+\psi(-1))q^{\binom{n+1}{2}}\sum_{\substack{0\leq r \leq
n\\r~even}}q^{rn-{\frac{1}{4}}r^2}\begin{bmatrix}
                                                              n \\
                                                              r \\
                                                            \end{bmatrix}
_q\prod_{j=1}^{r/2}(q^{2j-1}-1)
K_{GL(n-r,q)}(\psi;1)\\
&(=(\psi(1)+\psi(-1))\sum_{w\in Sp(2n,q)}\psi(Trw)).
\end{split}
\end{align*}
\indent

For our purposes, we only need the following expressions of Gauss
sums for $SO(3,q)$ and $O(3,q)$. So we state them separately as a
theorem. Also, for the ease of notations, we introduce

\begin{align*}
G_1(q)=SO(3,q),~G_2(q)=O(3,q).
\end{align*}
\indent

\begin{theorem}\label{C}
Let $\psi$ be any nontrivial additive character of $\mathbb{F}_q$.
Then we have
\begin{align*}
\begin{split}
&\sum_{w\in G_1(q)}\psi(Trw)=\psi(1)qK(\psi;1),\\
&\sum_{w\in G_2(q)}\psi(Trw)=(\psi(1)+\psi(-1))qK(\psi;1).
\end{split}
\end{align*}
\end{theorem}
\indent

The next corollary follows from Theorem \ref{B} and by simple change
of variables.

\begin{corollary}\label{D}
Let $\lambda$ be the canonical additive character of $\mathbb{F}_q$,
and let $a\in \mathbb{F}_{q}^{*}$. Then we have
\begin{equation}\label{u}
\sum_{w\in G_1(q)}\lambda(aTrw)=\lambda(a)qK(\lambda;a^2),\qquad
\qquad\quad\quad\quad
\end{equation}
\begin{equation}\label{v}
\begin{split}
\sum_{w\in G_2(q)}\lambda(aTrw)&=(\lambda(a)+\lambda(-a))qK(\lambda;a^2)\\
&=2(Re\lambda(a))qK(\lambda;a^2)~(cf.~(\ref{s})).
\end{split}
\end{equation}
\end{corollary}
\indent

\begin{proposition}\label{E}$($\cite{D4}, $($5.3-5$)$$)$
Let $\lambda$ be the canonical additive character of $F_q$, $m\in
\mathbb{Z}_{\geq0}$, $\beta\in\mathbb{F}_q$. Then
\begin{equation}\label{w}
\sum_{a\in\mathbb{F}_{q}^{*}}\lambda(-a\beta)K(\lambda;a^2)^m=q\delta(m,q;\beta)-(q-1)^m,
\end{equation}
where, for $m\geq1$,
\begin{equation}\label{x}
\delta(m,q;\beta)=|\{(\alpha_{1}^{},\cdots,\alpha_{m}^{})\in(\mathbb{F}_{q}^{*})^m|\alpha_{1}^{}+\alpha_{1}^{-1}+
\cdots,\alpha_{m}^{}+\alpha_{m}^{-1}=\beta\}|,
\end{equation}
and
\begin{align*}
\delta(0,q;\beta)=
\begin{cases}
1,& \text {$\beta$=0,}\\
0,& \text {otherwise.}
\end{cases}
\end{align*}
\end{proposition}
\indent

\begin{remark}\label{F}
Here one notes that
\begin{equation}\label{y}
\begin{split}
\delta(1,q;\beta)&=|\{x\in\mathbb{F}_{q}^{}|x^2-\beta x+1=0\}|\\
&=
\begin{cases}
2,& \text {if $\beta^2-1\neq0$ is a square,}\\
1,& \text {if $\beta^2-1$=0,}\\
0,& \text {if $\beta^2-1$ is a nonsquare.}
\end{cases}
\end{split}
\end{equation}
\end{remark}
\indent

Let $G(q)$ be  one of finite classical groups over $\mathbb{F}_q$.
Then we put, for each $\beta\in\mathbb{F}_q$,

\begin{align*}
N_{G(q)}(\beta)=|\{w\in G(q)|Tr(w)=\beta\}|.
\end{align*}
Then it is easy to see that

\begin{equation}\label{z}
qN_{G(q)}(\beta)=|G(q)|+\sum_{a\in\mathbb{F}_{q}^{*}}\lambda(-a\beta)
\sum_{w\in G(q)}\lambda(aTrw).
\end{equation}
For brevity, we write

\begin{equation}\label{a1}
n_1(\beta)=N_{G_1(q)}(\beta),~n_2(\beta)=N_{G_2(q)}(\beta).
\end{equation}

Using (\ref{u})-(\ref{w}), and (\ref{d1}), one derives the
following.
\indent

\begin{proposition}\label{G}
With the notations in (\ref{x}), (\ref{y}), and (\ref{a1}), we have:
\begin{equation}\label{b1}
n_1(\beta)=q^2-q+q\delta(1,q;\beta-1),\qquad \qquad \qquad \qquad
\end{equation}
\begin{equation}\label{c1}
n_2(\beta)=2q^2-2q+q\delta(1,q;\beta-1)+q\delta(1,q;\beta+1).
\end{equation}
\end{proposition}

\begin{corollary}\label{H}
$Tr:G_1(q)\rightarrow \mathbb{F}_q$, and $Tr:G_2(q)\rightarrow
\mathbb{F}_q$ are surjective.
\end{corollary}
\begin{proof}
This is immediate from the above Proposition \ref{G}.
\end{proof}
\indent

\section{Construction of codes}
\indent

Let
\begin{equation}\label{d1}
N_1=|G_1(q)|=q(q^2-1),~N_2=|G_2(q)|=2q(q^2-1).
\end{equation}
Here we will construct  ternary linear codes $C(G_1(q))$ of length
$N_1$ and $C(G_2(q))$ of length $N_2$, respectively associated with
the orthogonal groups $G_1(q)$ and $G_2(q)$. By abuse of notations,
let $g_1,g_2,\cdots,g_{N_i}$ be a fixed ordering of the elements in
the group $G_i(q)$, for $i=1,2$.\\
Also, we put
\begin{align*}
v_i=(Trg_1,Trg_2,\cdots,Trg_{N_i})\in\mathbb{F}_q^{N_i},~for~i=1,2.
\end{align*}
Then the ternary linear code  is defined as
\begin{equation}\label{e1}
C(G_i(q))=\{u\in\mathbb{F}_3^{N_i}|u\cdot v_i=0\},~for~i=1,2,
\end{equation}
where the dot denotes the usual inner product in
$\mathbb{F}_q^{N_i}$.
\indent

The following theorem of Delsarte is well-known.
\indent

\begin{theorem}\label{I} $($\cite{MS1}$)$
Let B be a linear code over $\mathbb{F}_q$. Then
\begin{align*}
(B|_{\mathbb{F}_3})^\bot=tr(B^\bot).
\end{align*}
In view of this theorem, the dual $C(G_i(q))^\bot$ is given by
\begin{equation}\label{f1}
C(G_i(q))^\bot=\{c_i(a)=(tr(aTrg_1),\cdots,tr(aTrg_{N_i}))|a\in\mathbb{F}_q\},~for~i=1,2.
\end{equation}
\end{theorem}

\begin{proposition}\label{J}
For every $q=3^r$, the map $\mathbb{F}_q\rightarrow
C(G_i(q))^\bot(a\mapsto c_i(a))$ is an $\mathbb{F}_3$-linear
isomorphism, for $i=1,2$.
\end{proposition}
\begin{proof}
The maps are clearly $\mathbb{F}_3$-linear and surjective. Let $a$
be in the kernel of either of the map. Then, in view of Corollary
\ref{H}, $tr(a\beta)=0$, for all $\beta\in\mathbb{F}_q$. Since the
trace function $\mathbb{F}_q \rightarrow \mathbb{F}_2$ is
surjective, $a=0$.
\end{proof}
\indent

\section{Power moments of Kloosterman sums with trace nonzero square arguments}
\indent

In this section, we will be able to find, via Pless power moment
identity,  recursive formulas for the power moments of Kloosterman
sums with trace nonzero square arguments in terms of the frequencies
of weights in $C(SO(3,q))$ and $C(O(3,q))$.

\begin{theorem}\label{K}
$($Pless power moment identity, \cite{MS1}$)$
 Let $B$ be an $q$-ary $[n,k]$ code, and let $B_i$(resp. $B_i^\bot$)
 denote the number of codewords of weight $i$ in $B$(resp. in
 $B^\bot$). Then, for $h=0,1,2,\cdots,$
\begin{equation}\label{g1}
\sum_{j=0}^{n}j^hB_j=\sum_{j=0}^{min\{n,h\}}(-1)^jB_j^\bot
\sum_{t=j}^{h}t!S(h,t)q^{k-t}(q-1)^{t-j}{\binom{n-j}{n-t}},
\end{equation}
where $S(h,t)$ is the Stirling number of the second kind defined in
(\ref{i}).
\end{theorem}

\begin{lemma}\label{LL}
Let $c_i(a)=(tr(aTrg_1),\cdots,tr(aTrg_{N_i}))\in C(G_i(q))^\bot$,
for $a\in\mathbb{F}_{q}^{*}$, and $i=1,2$. Then the Hamming weight
$w(c_i(a))$ can be expressed as follows:
\begin{equation}\label{h1}
w(c_i(a))={\frac{2qi}{3}}(q^2-1-(Re\lambda(a))K(\lambda;a^2)),~for~i=1,2.
\end{equation}
\begin{proof} For $i=1,2,$
\begin{align*}
\begin{split}
w(c_i(a))&=\sum_{j=1}^{N_i}(1-{\frac{1}{3}}\sum_{\alpha\in\mathbb{F}_3}
\lambda_0(\alpha tr(aTrg_j)))\\
&=N_i-{\frac{1}{3}}\sum_{\alpha\in\mathbb{F}_3}\sum_{w\in
G_i(q)}\lambda(\alpha aTrw)\\
&={\frac{2}{3}}N_i-{\frac{1}{3}}\sum_{\alpha\in\mathbb{F}_3^*}\sum_{w\in
G_i(q)}\lambda(\alpha aTrw).
\end{split}
\end{align*}
Our results now follow from (\ref{s}), (\ref{u}), (\ref{v}) and
(\ref{d1}).
\end{proof}
\end{lemma}
\indent

Fix $i(i=1,2)$, and let
$u=(u_1,\cdots,u_{N_i})\in\mathbb{F}_3^{N_i}$, with $\nu_\beta$ 1's
and $\mu_\beta$ 2's in the coordinate places where $Tr(g_j)=\beta$,
for each $\beta\in\mathbb{F}_{q}^{}$. Then we see from the
definition of the code $C(G_i(q))$(cf. (\ref{e1})) that $u$ is a
codeword with weight $j$ if and only if
$\sum_{\beta\in\mathbb{F}_q}\nu_\beta+\sum_{\beta\in\mathbb{F}_q}\mu_\beta=j$
and $\sum_{\beta\in\mathbb{F}_q}\nu_\beta
\beta=\sum_{\beta\in\mathbb{F}_q}\mu_\beta \beta$(an identity in
$\mathbb{F}_q$). Note that there are
$\prod_{\beta\in\mathbb{F}_q}{\binom{n_i(\beta)}{\nu_\beta,\mu_\beta}}$(cf.
(\ref{d}), (\ref{e})) many such codewords with weight $j$. Now, we
get the following formulas in (\ref{i1})-(\ref{j1}), by using the
explicit values of $n_i(\beta)$ in (\ref{b1}),(\ref{c1})(cf.
(\ref{x}), (\ref{y})).

\begin{theorem}\label{M}
Let $q=3^r$ be as before, and let $\{C_{i,j}\}_{j=0}^{N_i}$ be the
weight distribution of $C(G_i(q))$, for $i=1,2.$ Then

$(1)$
\begin{equation}\label{i1}
C_{1,j}=\sum\prod_{\beta\in\mathbb{F}_q}{\binom{n_1(\beta)}{\nu_\beta,\mu_\beta}}~
(j=0,\cdots,N_1),\qquad \qquad \qquad \qquad
\end{equation}
with
\begin{equation*}
\begin{split}
n_1(\beta)&=q^2-q+q\delta(1,q;\beta-1)\\
&=\sum{\binom{q^2}{\nu_0,\mu_0}}{\binom{q^2}{\nu_2,\mu_2}}
\prod_{\beta^2-2\beta\neq0~square}{\binom{q^2+q}{\nu_\beta,\mu_\beta}}\\
&\quad\qquad\qquad\qquad\qquad\qquad\qquad\times\prod_{\beta^2-2\beta~nonsquare}
{\binom{q^2-q}{\nu_\beta,\mu_\beta}},
\end{split}
\end{equation*}

$(2)$
\begin{equation}\label{j1}
C_{2,j}=\sum\prod_{\beta\in\mathbb{F}_q}{\binom{n_2(\beta)}{\nu_\beta,\mu_\beta}}~
(j=0,\cdots,N_2),\qquad \qquad \qquad \qquad
\end{equation}
with
\begin{equation*}
n_2(\beta)=2q^2-2q+q\delta(1,q;\beta-1)+q\delta(1,q;\beta+1).\qquad\qquad\quad
\end{equation*}
\end{theorem}
Here in both (\ref{i1}) and (\ref{j1}) the unspecified sums run over
all the sets of nonnegative integers
$\{\nu_\beta\}_{\beta\in\mathbb{F}_q}$ and
$\{\mu_\beta\}_{\beta\in\mathbb{F}_q}$ satisfying
\begin{align*}
\sum_{\beta\in\mathbb{F}_q} \nu_\beta
+\sum_{\beta\in\mathbb{F}_q}\mu_\beta=j\quad and~
\sum_{\beta\in\mathbb{F}_q} \nu_\beta
\beta=\sum_{\beta\in\mathbb{F}_q}\mu_\beta \beta,
\end{align*}
and, for every $\beta\in\mathbb{F}_q$, $\delta(1,q;\beta)$ is as in (\ref{y}).\\

The recursive formula in the following theorem follows from the
study of ternary linear codes associated with the symplectic group
$Sp(2,q)=SL(2,q)$. It is slightly modified from its original
version, which makes it more usable in below. \\

\begin{theorem}\label{L}$($\cite{DJ1}$)$
For $h=1,2,3,\cdots$,
\begin{equation}\label{k1}
\begin{split}
2(\frac{2q}{3})^h\sum_{j=0}^{h}&(-1)^j{\binom{h}{j}}(q^2-1)^{h-j}SK^{j}\\
&=q\sum_{j=0}^{min\{N_1,h\}}(-1)^j
\hat{C}_{j}\sum_{t=j}^{h}t!S(h,t)3^{-t}2^{t-j}
{\binom{N_1-j}{N_1-t}},
\end{split}
\end{equation}
where $N_1=q(q^2-1)=|Sp(2,q)|=|SO(3,q)|$, $S(h,t)$ indicates the
Stirling number of the second kind as in (\ref{i}), and
$\{\hat{C}_j\}_{j=0}^{N_1}$ denotes the weight distribution of the
ternary linear code $C(Sp(2,q))$, given by
\begin{align*}
\begin{split}
\hat{C}_j&=\sum_{\beta\in\mathbb{F}_q}\prod_{\beta}{\binom{q\delta(1,q;\beta)+q^2-q}{\nu_\beta,\mu_\beta}}\\
&=\sum{\binom{q^2}{\nu_1,\mu_1}}{\binom{q^2}{\nu_{-1},\mu_{-1}}}
\prod_{\beta^2-1\neq0~square}{\binom{q^2+q}{\nu_\beta,\mu_\beta}}\prod_{\beta^2-1~nonsquare}{\binom{q^2-q}{\nu_\beta,\mu_\beta}}
\end{split}
\end{align*}
$(j=0,\cdots,N_1)$.\\
Here the sum is over all the sets of nonnegative integers
$\{\nu_\beta\}_{\beta\in\mathbb{F}_q}$ and
$\{\mu_\beta\}_{\beta\in\mathbb{F}_q}$ satisfying
$\sum_{\beta\in\mathbb{F}_q} \nu_\beta
+\sum_{\beta\in\mathbb{F}_q}\mu_\beta=j$ and
$\sum_{\beta\in\mathbb{F}_q} \nu_\beta
\beta=\sum_{\beta\in\mathbb{F}_q}\mu_\beta \beta$.
\end{theorem}
\indent

We are now ready to apply the Pless power moment identity in
(\ref{g1}) to $C(G_i(q))^\bot$ for $i=1,2$, in order to obtain the
result in Theorem \ref{A}(cf. (\ref{f})-(\ref{h}),
(\ref{j})-(\ref{l})) about recursive formulas. We do this for
$i=1,2$ at the same time.\\

The left hand side of that identity in (\ref{g1}) is equal to
\begin{equation}\label{l1}
\sum_{a\in\mathbb{F}_{q}^{*}}w(c_i(a))^h,
\end{equation}
with the $w(c_i(a))$ given by (\ref{h1}).\\

In below, $\lq\lq$the sum over $tra=0$ (resp. $tra\neq0$)" will mean
$\lq\lq$the
sum over all $a\in\mathbb{F}_{q}^{*}$ with $tra=0$(resp. $tra\neq0$)."\\

(\ref{l1}) is given by

\begin{align*}
\begin{split}
({\frac{2qi}{3}})^h&\sum_{a\in\mathbb{F}_{q}^{*}}(q^2-1-(Re\lambda(a))K(\lambda;a^2))^h\\
=&({\frac{2qi}{3}})^h\sum_{tra=0}(q^2-1-K(\lambda;a^2))^h\\
&\qquad+({\frac{2qi}{3}})^h\sum_{tra\neq0}(q^2-1+{\frac{1}{2}}K(\lambda;a^2))^h\\
(noting&~ that~ Re\lambda(a)=1,~ if~ tra=0; Re\lambda(a)=-\frac{1}{2},~ if~ tra\neq0, i.e., tra=1,2)\\
=&({\frac{2qi}{3}})^h\sum_{tra=0}\sum_{j=0}^h(-1)^j{\binom{h}{j}}(q^2-1)^{h-j}K(\lambda;a^2)^j\\
&\qquad+({\frac{2qi}{3}})^h\sum_{tra\neq0}\sum_{j=0}^h{\binom{h}{j}}(q^2-1)^{h-j}2^{-j}K(\lambda;a^2)^j\\
=&({\frac{2qi}{3}})^h\sum_{j=0}^h(-1)^j{\binom{h}{j}}(q^2-1)^{h-j}(2SK^j-T_{12}SK^j)\\
&\qquad\qquad\qquad\qquad\qquad\qquad\qquad\text{($\psi=\lambda$ case of (\ref{a}), (\ref{b}))}\\
&\qquad+({\frac{2qi}{3}})^h\sum_{j=0}^h{\binom{h}{j}}(q^2-1)^{h-j}2^{-j}T_{12}SK^j\\
=&i^h2({\frac{2q}{3}})^h\sum_{j=0}^h(-1)^j{\binom{h}{j}}(q^2-1)^{h-j}SK^j\\
&\qquad+({\frac{2qi}{3}})^h\sum_{j=0}^h((-1)^{j+1}+2^{-j}){\binom{h}{j}}(q^2-1)^{h-j}T_{12}SK^j\\
\end{split}
\end{align*}
\begin{equation}\label{m1}
\begin{split}
=i^hq&\sum_{j=0}^{min\{N_1,h\}}(-1)^j
\hat{C}_{j}\sum_{t=j}^{h}t!S(h,t)3^{-t}2^{t-j}
{\binom{N_1-j}{N_1-t}}~\text{(from (\ref{k1}))}\hspace{8mm}\\
&+({\frac{2qi}{3}})^h\sum_{j=0}^h((-1)^{j+1}+2^{-j}){\binom{h}{j}}(q^2-1)^{h-j}T_{12}SK^j.
\end{split}
\end{equation}

On the other hand, the right hand side of  (\ref{g1}) is
\begin{equation}\label{n1}
q\sum_{j=0}^{min\{N_1,h\}}(-1)^j
C_{i,j}\sum_{t=j}^{h}t!S(h,t)3^{-t}2^{t-j} {\binom{N_i-j}{N_i-t}}.
\end{equation}
\indent Here one has to note that
$dim_{\mathbb{F}_2}C(SO(3,q))=dim_{\mathbb{F}_2}C(O(3,q))=r$(cf.
Prop. \ref{J}) and to separate the term corresponding to $l=h$ of
the second sum in (\ref{m1}). Our main results in Theorem \ref{a}
now follow by equating (\ref{m1}) and (\ref{n1}).\\

\begin{corollary}\label{N}
Let $q=3^r$. Then we have the following.
\begin{enumerate}
\item\label{Na} $SK={\frac{1}{2}}\{(-1)^rq+1\}$,
\item\label{Nb} $T_0SK={\frac{1}{3}}(-1)^rq+1$,
\item\label{Nc} $T_{12}SK={\frac{2}{3}}(-1)^r q$.
\end{enumerate}
\begin{proof}
From either (\ref{f}) or (\ref{j}), we get (\ref{Nc}).  (\ref{Na})
follows from our previous result (\cite{DJ1}, (4)) or can be derived
directly as follows.
\begin{align*}
\begin{split}
SK&={\frac{1}{2}}\sum_{a\in\mathbb{F}_{q}^{*}}K(\lambda;a^2)\\
&={\frac{1}{2}}\sum_{a\in\mathbb{F}_{q}^{*}}\sum_{\alpha\in\mathbb{F}_{q}^{*}}
\lambda(\alpha_{}^{} +a^2\alpha_{}^{-1})\\
&={\frac{1}{2}}\sum_{a\in\mathbb{F}_{q}^{*}}\sum_{\alpha\in\mathbb{F}_{q}^{*}}
\lambda(a(\alpha_{}^{} +\alpha_{}^{-1}))\\
\end{split}
\end{align*}
\begin{equation}\label{o1}
\qquad\qquad\quad\qquad={\frac{1}{2}}\sum_{\alpha\in\mathbb{F}_{q}^{*}}\sum_{a\in\mathbb{F}_{q}^{}}
\lambda(a(\alpha_{}^{} +\alpha_{}^{-1}))-{\frac{1}{2}}(q-1)
\end{equation}
\begin{equation*}
\hspace{11.5mm}=
\begin{cases}
\frac{1}{2}2q-\frac{1}{2}(q-1),& \text {if $r$ even},\\
-\frac{1}{2}(q-1),& \text {if $r$ odd}.
\end{cases}
\end{equation*}
In (\ref{o1}), we note that $\alpha+\alpha^{-1}=0$ has a solution in
$\mathbb{F}_{q}^{*}$ if and only if -1 is a square in $\mathbb{F}_q$
if and only if $r$ is even, in which case there are two distinct
solutions. Finally, (\ref{Nb}) follows from the relation (\ref{b})
with $h=1$ and $\psi=\lambda$.
\end{proof}
\end{corollary}

\bibliographystyle{amsplain}

\begin{thebibliography}{10}


\bibitem {E1} R.J. Evans, \textit{Seventh power moments of Kloosterman sums}, Israel J. Math.,
\textbf{}  to appear.

\bibitem {GS1} G. van der Geer, R. Schoof and M. van der Vlugt, \textit{Weight formulas for ternary Melas codes}, Math. Comp.
Math. \textbf{58}(1992), 781--792.

\bibitem {HS1} K. Hulek, J. Spandaw, B. van Geemen and D. van Straten, \textit{The modulartiy of the Barth-Nieto quintic and its relatives},
Adv. Geom. \textbf{1} (2001), 263--289.

\bibitem {D3} D. S. Kim, \textit{Gauss sums for symplectic groups over a finite field}, Mh. Math. \textbf{126} (1998), 55--71.

\bibitem {D4} D. S. Kim, \textit{Exponential sums for symplectic groups and their apllications}, Acta Arith. \textbf{88} (1999), 155--171.

\bibitem {D1} D. S. Kim, \textit{Gauss sums for $O^{}(2n+1,q)$}, Finite Fields Appl. \textbf{4} (1998), 62--86.

\bibitem {D2} D. S. Kim, \textit{Codes associated with $O(3,2^r)$ and power moments of Kloosterman sums with trace one arguments},
 \textbf{} submitted.

\bibitem {D5} D. S. Kim, \textit{An infinite family of recursive formulas generating power moments of  Kloosterman sums with trace one
arguments: $O(2n+1,2^r)$ case},
 \textbf{} submitted.

\bibitem {DJ1} D. S. Kim and J. H. Kim, \textit{Ternary codes associated with symplectic groups and power moments of Kloosterman
sums with square arguments}, submitted.

\bibitem {K1} H. D. Kloosterman, \textit{On the representation of numbers in the form $ax^2+by^2+cz^2+dt^2$ },
 Acta Math. \textbf{49} (1926), 407-464.

\bibitem {LN1} R. Lidl and H. Niederreiter, \textit{Finite Fields}, Encyclopedia Math.
Appl.\textbf{20},  Cambridge University Pless, Cambridge,  1987.

\bibitem {L1} R. Livn$\acute{e}$, \textit{Motivic orthogonal two-dimensional representations of
$Gal(\overline{\mathbb{Q}}/\mathbb{Q})$}, Israel J. Math.
\textbf{92} (1995), 149-156.

\bibitem {MS1} F. J. MacWilliams and N. J. A. Sloane, \textit{The Theory of Error Correcting Codes},
North-Holland, Amsterdam, 1998.

\bibitem {M1} M. Moisio, \textit{On the moments of Kloosterman sums and fibre products of Kloosterman curves},
Finite Fields Appl. \textbf{14}(2008), 515--531.

\bibitem {PV1} C. Peters, J. Top, and M. van der Vlugt, \textit{The Hasse zeta function of a K3
surface related to the number of words of weight 5 in the Melas
codes}, J. Reine Angew. Math. \textbf{432} (1992), 151-176.


\bibitem {S1} H. Sali$\acute{e}$, \textit{$\ddot{U}$ber die Kloostermanschen Summen $\mathcal{S}(u,v;q)$}, Math. Z. \textbf{34}(1931), 91-109.

\end{thebibliography}

\end{document}